\documentclass{amsart}
\usepackage{amssymb,amsfonts,amscd}
\usepackage{amsmath}
\usepackage[arrow, matrix, curve]{xy}
\usepackage{eucal}
\usepackage{dsfont}
\usepackage{hyperref}

\DeclareFontFamily{U}{matha}{\hyphenchar\font45}
\DeclareFontShape{U}{matha}{m}{n}{
      <5> <6> <7> <8> <9> <10> gen * matha
      <10.95> matha10 <12> <14.4> <17.28> <20.74> <24.88> matha12
      }{}
\DeclareSymbolFont{matha}{U}{matha}{m}{n}
\DeclareFontSubstitution{U}{matha}{m}{n}
\DeclareMathSymbol{\operp}         {2}{matha}{"6B}

\newcommand{\bp}{{\bar{P}}}
\newcommand{\D}{{\mathcal{D}}}
\newcommand{\V}{{\mathcal{V}}}
\newcommand{\T}{{\mathcal{T}}}
\newcommand{\C}{{\mathcal{C}}}
\newcommand{\K}{{\mathcal{K}}}
\newcommand{\R}{\mathds{R}}

\newcommand{\mx}{\mathfrak{X}}
\newcommand{\lie}[1]{\mathfrak{#1}}
\newcommand{\dr}{\mathbf{d}}

\newcommand{\ip}[1]{{\mathbf{i}}_{#1}}
\newcommand{\an}[1]{\arrowvert_{#1}}

\newcommand{\pairing}{\langle\cdot\,,\cdot\rangle}

\DeclareMathOperator{\dom}{Dom}

\DeclareMathOperator{\pr}{pr}
\newcommand{\inv}{^{-1}}

\theoremstyle{plain}
\newtheorem{theorem}{Theorem}
\newtheorem{lemma}[theorem]{Lemma}
\newtheorem{proposition}[theorem]{Proposition}

\theoremstyle{definition}

\theoremstyle{remark}
\newtheorem{remark}{Remark}

\begin{document}
\date{}
\title{Induced Dirac structures on isotropy type manifolds}
\author{M. Jotz}
\address{Section de Math{\'e}matiques\\
Ecole Polytechnique
  F{\'e}d{\'e}rale de Lausanne\\
1015 Lausanne, Switzerland\\
madeleine.jotz@epfl.ch}
\author{
T.S. Ratiu}
\thanks{The authors were partially supported by Swiss NSF grant 200021-121512}
\address{Section de Math{\'e}matiques
 and Bernouilli Center\\
Ecole Polytechnique
  F{\'e}d{\'e}rale de Lausanne\\
1015 Lausanne, Switzerland\\
tudor.ratiu@epfl.ch
}
\maketitle

\begin{abstract}
A new method of singular reduction is extended from Poisson
to Dirac manifolds. Then it is shown that the Dirac structures on
the strata of the quotient coincide with those of the only other known 
singular Dirac reduction method.

\noindent \textbf{AMS Classification:} Primary subjects: 70H45, 70G65, 53C15, 53C10
\quad Secondary subjects: 70G45, 53D17, 53D99

\noindent \textbf{Keywords:} Dirac structures, singular reduction,
proper action.
\end{abstract}

\section{Introduction}

Dirac structures were introduced in \cite{CoWe88} and studied
systematically for the first time in \cite{Courant90a}. The 
past years have seen remarkable applications of Dirac 
manifolds in geometry and  theoretical physics. Dirac 
structures include two-forms, Poisson structures, 
and foliations. They turn out to provide the right geometric 
framework for nonholonomic systems and circuits. If a Lie 
group action is compatible with the Dirac structure on a manifold, one has all ingredients for reduction;
see  \cite{Courant90a}, \cite{Blankenstein00}, 
\cite{BlvdS01}, \cite{BuCaGu07}, \cite{BuCr05}, 
\cite{JoRa08}, \cite{Zambon08}, \cite{JoRaZa11}, 
\cite{MaYo07}, \cite{MaYo09}, for the regular case, 
\cite{BlRa04}, \cite{JoRaSn11} for the singular situation, 
and \cite{JoRa10b} for optimal reduction. All these 
reduction procedures are in the spirit of Poisson reduction (\cite{MaRa86}, \cite{Sniatycki03}, \cite{FeOrRa09}, 
\cite{JoRa09}).

Consider a Dirac structure $\mathsf{D}$  on a manifold
$M$ which is invariant under the  proper action of a Lie 
group $G$. The most general singular Dirac reduction method 
known today was
introduced in \cite{JoRaSn11} and is in the same spirit as
\cite{Sniatycki03}. Certain mild regularity conditions are 
required to construct a specific subspace of the set of 
pairs of \emph{vector fields} with \emph{one-forms} 
on the stratified space $\bar{M}:=M/G$ which is then shown 
to naturally induce  a Dirac structure on each stratum of 
$\bar{M}$. These regularity conditions are automatically 
satisfied for Poisson manifolds.

However, in the case of Poisson manifolds, there is an 
alternative singular reduction method presented in 
\cite{FeOrRa09}. If $(M,\{\cdot\,,\cdot\})$ is a Poisson 
manifold with a Lie group $G$ acting in a proper canonical 
way on it, then there is an induced Poisson structure on  
each isotropy type manifold $M_H$ (see also \cite{JoRa09}) 
that is invariant under the induced action of $N(H)/H$; 
here $H$ is an isotropy subgroup of the action and $N(H)$ 
is the normalizer of $H $ in $G $. Therefore, these 
Poisson structures descend to the quotient $M_H/(N(H)/H)$ 
whose connected components are strata in $\bar{M}$. 

The first goal of this paper is to extend this reduction method to
symmetric Dirac manifolds. In a first step, the
$G$-invariant Dirac structure $\mathsf{D}$ on $M$ 
is shown to induce a 
Dirac structure $\mathsf{D}_Q$ on each connected 
component $Q$ of an isotropy type manifold $M_H$. Then, 
$\mathsf{D}_Q$ is shown to satisfy 
the conditions for regular Dirac reduction (as in \cite{JoRaZa11}), 
thus descending to a Dirac structure on the stratum $Q/N(H)$ of $\bar{M}$.
The second goal of the paper is to show that the reduced manifolds 
obtained this way correspond exactly to the singularly reduced 
manifolds described previously (those obtained in \cite{JoRaSn11}).
More precisely, the Dirac structures induced on on the connected
components of $M_{(H)}/G$  and of $M_H/N(H)$ are forward and 
backward Dirac images of each other under a canonical diffeomorphism 
between the two reduced manifolds.

\subsubsection{Acknowledgments} The authors would like to 
thank Rui Loja Fernandes for a conversation that led us to consider the subject of 
the present paper. His advice and challenge to compare the two singular
reduction methods are greatly appreciated.
Many thanks go also to the referees for many useful 
comments that have improved the exposition.

\subsubsection{Notation and conventions}
The manifold $M$ is always assumed to be paracompact. 
The sheaf of local functions on $M$ is denoted by $C^\infty(M)$; that is,
an element $f \in C^\infty(M)$ is a smooth function $f:U\to \R$, with $U$ an
open subset of $M$. Similarly, if $E$ is a vector bundle or a generalized 
distribution over $M$, $\Gamma(E)$ denotes the set of local sections 
of $E$. In particular, $\mathfrak{X}(M)$ and $\Omega^1(M)$ are the sets 
of local vector fields and one-forms on $M$, respectively. The open  
domain of definition  of  $\sigma\in \Gamma(E)$ is denoted by 
$\dom(\sigma)$.

\medskip

The Lie group $G$ is always assumed to be connected;
$\mathfrak{g}$ denotes the Lie algebra of $G$. Let $\Phi: 
(g, m) \in G\times M \mapsto gm = g \cdot m=\Phi_g(m) \in M$, 
be a smooth left action. If $\xi \in \mathfrak{g}$ then 
$\xi_M \in \mathfrak{X} (M) $,
defined by $\xi_M(m) := \left.\frac{d}{dt}\right|_{t=0} \exp (t \xi) \cdot m$, 
is called the \emph{infinitesimal generator} or 
\emph{fundamental vector field} defined by $\xi$.

\medskip

A section $X$ of $TM$ (respectively $\alpha$ of  $T^*M$) is
called \emph{$G$-invariant} if $\Phi_g^*X=X$ (respectively
$\Phi_g^*\alpha=\alpha$) for all $g\in G$. Recall that  $\Phi_g^\ast X:=
T\Phi_{g^{-1}}\circ X\circ\Phi_g$, that is, $(\Phi_g^\ast X)(m)=T_{gm}\Phi_{g^{-1}}X(gm)$ for all 
$m\in M$.

\medskip

We write $TM\operp TM^*$ for the direct sum of the vector  
bundles  $TM$ and $T^*M$ and use the same notation for the sum 
of a  tangent (a distribution in $TM$) and cotangent 
distribution (a distribution in  $T^*M$; 
see section \ref{dis} for the definitions of those objects). We choose
this notation because we want to distinguish these direct sums from  
direct sums of distributions in a same vector bundle, denoted, as usual, 
by $\oplus$.

\section{Generalities on Dirac structures and distributions}
\label{sec:Dirac_structures}

\subsection{Generalized distributions and orthogonal spaces}\label{dis}

Let $E$ be a vector bundle over $M$.
A  \emph{generalized distribution $ \Delta $ in $E$} is a subset
$\Delta$ of $E$  such that for each $m\in M$, the set $ \Delta (m) : 
= \Delta \cap E(m)$ is a
vector subspace of $E_m$. The number $\dim \Delta (m) $ is called the
\emph{rank}  of $ \Delta $ at $ m  \in  M $. A point $m\in M$ is a
\emph{regular} point of the generalized distribution $\Delta$ if there exists a
neighborhood $U$ of $m$ such that the rank of $\Delta$ is constant on
$U$. Otherwise, $m$ is a \emph{singular} point of the generalized distribution. 

A local \emph{differentiable
  section} of $ \Delta $ is a smooth section $\sigma\in \Gamma(E)$
defined on some open subset $ U \subset  M $ such that $ \sigma(u) \in  \Delta (u)
$ for each $ u \in  U $. We denote with $ \Gamma(
\Delta ) $ the space
of local sections of $ \Delta $. A generalized
distribution is said to be \emph{differentiable} or \emph{smooth} if for every
point $ m \in  M $ and every vector $ v \in  \Delta (m) $, there is a
differentiable section $ \sigma \in  \Gamma ( \Delta ) $ defined on an open
neighborhood $ U $ of $ m $ such that $ \sigma(m) = v $. 

A smooth generalized distribution in  the tangent bundle $TM$ 
(that is, $E=TM$) is called a \emph{smooth tangent distribution}. 
A smooth generalized distribution in the cotangent bundle $T^*M$ 
(that is, $E= T ^\ast M $) is called a \emph{smooth cotangent distribution}. 
We will work most of the time with smooth generalized
distributions in the \textit{Pontryagin bundle} $\mathsf P_M:=TM\operp T^*M$ which we will
call \emph{smooth generalized distributions}, for simplicity.

\subsubsection{Generalized smooth distributions and annihilators}
Assume that $E$ is a vector bundle on $M$ that is endowed 
with a smooth non-degenerate
symmetric bilinear map $\pairing_E$.
In the special case where $E$ is the  \emph{Pontryagin bundle} $\mathsf{P}_M = TM \operp T^* M$ of  
a smooth manifold $M$, we will always consider the non-degenerate symmetric fiberwise bilinear form 
of signature $(\dim M, \dim M)$ given by
\begin{equation}\label{pairing}
\left\langle (u_m, \alpha_m), ( v_m, \beta_m ) \right\rangle : 
= \beta_m(u_ m ) + \alpha_m( v _m)
\end{equation}
for all $u _m, v _m \in T _mM$ and $\alpha_m, \beta_m \in T^\ast_mM$.

If $\Delta \subset E $ is a smooth distribution in $E$, its
\emph{smooth} \emph{orthogonal} distribution  is the smooth
generalized distribution $\Delta^\perp $ in $E$  defined by 
\begin{align*}
\Delta^ \perp (m): =
 \left\{\tau(m) \left|  
\begin{array}{c}\tau \in
\Gamma(E) \text{ with } m\in \dom(\tau) \text{ is such that  for all }\\
\sigma \in\Gamma(\Delta) \text{  with } m\in\dom(\sigma),\\
\text{ we have }  \left\langle \sigma,\tau
\right\rangle_E = 0 \text{ on } \dom(\tau)\cap\dom(\sigma)\end{array}\right.\right\}.
\end{align*}
We have $\Delta \subset \Delta^{ \perp\perp }$, in general strict.
 Note that the smooth orthogonal distribution of a smooth generalized
 distribution
 is smooth by construction. If the distribution $\Delta$ is
a vector subbundle of $E$, then its smooth orthogonal distribution is
also a vector subbundle of $E$. 
The proof of the following proposition can be found in \cite{JoRaSn11}.
\begin{proposition}\label{prop_intersection}
Let $\Delta_1$ and $\Delta_2$ be smooth subbundles of the vector bundle
$(E,\pairing_E)$.
Since both $\Delta_1$ and $\Delta_2$ have constant rank 
on $M$, 
their smooth orthogonals $\Delta_1^\perp$ 
and $\Delta_2^\perp$ are also smooth subbundles of $E$ and equal to the
pointwise orthogonals of $\Delta_1$ and $\Delta_2$. 
The following are equivalent:
\begin{itemize}
\item[{\rm (i)}] The intersection
$\Delta_1^\perp\cap\Delta_2^\perp$ is smooth.
\item[{\rm (ii)}] $ (\Delta_1+\Delta_2)^\perp=
\Delta_1^\perp\cap\Delta_2^\perp$
\item[{\rm (iii)}]
$(\Delta_1^\perp\cap\Delta_2^\perp)^\perp=\Delta_1+\Delta_2$
\item[{\rm (iv)}] $\Delta_1^\perp\cap\Delta_2^\perp$ has constant rank on $M$.
\end{itemize}
\end{proposition}

A tangent (respectively cotangent) distribution $\T\subseteq TM$ (respectively
$\mathcal{C}\subseteq T^*M$) can be identified with
the smooth generalized distribution $\T\operp\{0\}$ 
(respectively $\{0\}\operp\mathcal{C}$). 
The smooth orthogonal distribution of $\T\operp\{0\}$ in $TM\operp T^*M$ is 
easily computed to be 
$(\T\operp\{0\})^\perp=TM\operp \T^\circ$,
where
\[\T^\circ(m)=\left\{\alpha(m)\left|
\begin{array}{c}
 \alpha\in\Omega^1(M), m\in\dom(\alpha) \text{
  and } 
\alpha(X)=0\\ \text{ on }\dom(\alpha)\cap\dom(X)\text{ for all }
X\in\Gamma(\T)
\end{array}
\right.\right\}\subseteq T_m^*M
\]
for all $m\in M$. This smooth cotangent distribution is called the
\emph{smooth annihilator} of $\T$. Analogously, we define the smooth annihilator
$\C^\circ\subseteq TM$ of a
cotangent distribution $\C$. Then $\C^\circ$ is a smooth tangent distribution
and we have $(\{0\}\operp\C)^\perp=\C^\circ\operp T^*M$. 
\medskip
  
The tangent distribution $\V$ spanned by the fundamental vector fields of the
action of a Lie group $G$ on a manifold $M$ will be of great importance later on. 
At every point $m \in M $ it is defined by
\[\V(m)=\{\xi_M(m)\mid \xi\in\lie g\}.\]
If the action is not free, the rank of the fibers of $\V$
can vary on $M$. The smooth annihilator  $\V^\circ$ of $\V$ is given by
\[\V^\circ(m)=\{\alpha(m)\mid \alpha\in\Omega^1(M), \,m\in\dom(\alpha),\text{
  such that } \alpha(\xi_M)
=0 \text{ for
  all }\xi \in \lie g\}.\]
We will use also the smooth generalized distribution $\K:=\V\operp \{0\}$ and its
smooth orthogonal space $\K^\perp=TM\operp\V^\circ$.

We will need the smooth codistribution $\V_G^\circ$ spanned
by the $G$-invariant sections of $\V^ \circ$. It is shown in \cite{JoRaSn11}
that 
$\V_G^\circ(m)=\{\mathbf{d}f_m\mid f\in C^\infty(M)^G\}$
for all $m\in M$. Thus, $\V_G^\circ$ is spanned by the exact $G$-invariant
sections of $\V^\circ$.

\subsection{Dirac structures} \label{admsble}
Recall that the Pontryagin bundle $\mathsf P_M=TM\operp T^*M$ is endowed with a 
natural pairing
$\pairing$ given by \eqref{pairing}.
A \emph{Dirac structure} 
(see \cite{Courant90a}) on $M$ is a Lagrangian subbundle 
$\mathsf{D} \subset TM \operp T^* M $. That is, 
$\mathsf{D}$ coincides with its orthogonal relative to \eqref{pairing} 
and so its fibers are necessarily $\dim M$-dimensional.

The space $\Gamma(\mathsf P_M) $ of local sections of the Pontryagin
bundle is also endowed with a  skew-symmetric bracket, the \emph{Courant bracket}, given by
\begin{align}\label{wrong_bracket}
[(X, \alpha), (Y, \beta) ] : &
= \left( [X, Y],  {\boldsymbol{\pounds}}_{X} \beta - {\boldsymbol{\pounds}}_{Y} \alpha + \frac{1}{2}
  \mathbf{d}\left(\alpha(Y) 
- \beta(X) \right) \right)\nonumber \\
&= \left([X, Y],  {\boldsymbol{\pounds}}_{X} \beta - \mathbf{i}_Y \mathbf{d}\alpha 
- \frac{1}{2} \mathbf{d} \left\langle (X, \alpha), (Y, \beta) \right\rangle
\right)
\end{align}
for all $(X,\alpha), (Y,\beta)\in\Gamma(\mathsf P_M)$ 
(see \cite{Courant90a}). This bracket is $\mathds{R}$-bilinear 
(in the sense that
$[a_1(X_1, \alpha_1)+a_2(X_2, \alpha_2), (Y, \beta) ]
=a_1[(X_1, \alpha_1), (Y, \beta) ]+a_2[(X_2, \alpha_2), (Y, \beta)
]$ for all $a_1,a_2\in\R$ and
$(X_1, \alpha_1),(X_2, \alpha_2)$, $(Y, \beta)\in\Gamma(TM\oplus T^*M)$  on the common domain of
definition of the three sections), skew symmetric, but does not satisfy the Jacobi identity.

 The Dirac structure is \emph{integrable} (or \emph{closed})  if 
$[ \Gamma(\mathsf{D}), \Gamma(\mathsf{D}) ] \subset \Gamma(\mathsf{D}) $. Since 
$\left\langle (X, \alpha), (Y, \beta) \right\rangle = 0 $ if $(X, \alpha), (Y,
\beta) \in \Gamma(\mathsf{D})$, 
integrability of the Dirac structure is often expressed in the literature
relative to a non-skew symmetric
 bracket that differs from 
\eqref{wrong_bracket} by eliminating in the second formula the 
third term of the second component. This truncated expression
which satisfies the Jacobi identity if and only if the Dirac structure 
is integrable, but is no longer skew-symmetric, appears
in the literature sometimes also as \emph{Courant}, 
\emph{Courant-Dorfman}, or \emph{Dorfman bracket}:
\begin{equation}\label{Courant_bracket}
[(X, \alpha), (Y, \beta) ] : = \left( [X, Y],  {\boldsymbol{\pounds}}_{X} \beta - \ip{Y} \dr\alpha \right).
\end{equation}  

If the Dirac structure $(M,\mathsf D)$
is integrable, then $\mathsf D$ has the structure of a Lie algebroid over $M$ 
with anchor map the projection $\mathsf P_M\to TM$ and bracket the Courant bracket.

\subsubsection{Maps in the Dirac category}
Let $(M,\mathsf D_M)$ and $(N,\mathsf D_N)$ be smooth Dirac manifolds.
A smooth map $\phi:(M,\mathsf D_M)\to (N,\mathsf D_N)$
is  a \emph{forward Dirac map} if 
for all $(Y,\beta)\in\Gamma(\mathsf D_N)$
there exist $X\in\mx(M)$
such that 
$X\sim_\phi Y$ and $(X,\phi^*\beta)\in\Gamma(\mathsf D_M)$.
It is a \emph{backward Dirac map} if for all
$(X,\alpha)\in\Gamma(\mathsf D_M)$ there exist $(Y,\beta)\in\Gamma(\mathsf D_N)$
such that 
$X\sim_\phi Y$ and $\alpha=\phi^*\beta$.

If $\phi$ is a diffeomorphism, then it is easy to 
check that it is a backward Dirac map if and only if it is a forward Dirac map. 

Let $M$ and $N$ be smooth manifolds and 
 $\phi:M\to N$ a smooth map. Assume that $N$ is endowed 
with a Dirac structure $\mathsf D_N$. The \emph{pull back}
$\phi^*\mathsf D_N$ of $\mathsf D_N$ is
the subdistribution of $\mathsf P_M$ defined 
by 
\begin{equation}\label{pullback}
(\phi^*\mathsf D_N)(m)=\left\{
(v_m,\alpha_m)\in\mathsf P_M(m)\left| 
\begin{array}{c}
\exists (w_{\phi(m)},\beta_{\phi(m)})\in\mathsf D_N(\phi(m))\\
\text{ such that } T_m\phi(v_m)=w_{\phi(m)} \\\text{ and }
\alpha_m=(T_m\phi)^*\beta_{\phi(m)}
\end{array}\right.\right\}
\end{equation}
for all $m\in M$. 

Each fiber of the subdistribution $\phi^*\mathsf D_N$ is 
Lagrangian, that is, $((\phi^*\mathsf D_N)(m))^\perp=
(\phi^*\mathsf D_N)(m)$ in $T_mM\times T_m^*M$ for all 
$m\in M$. This fact is well known; we will prove it 
here for the sake of completeness. The inclusion
$(\phi^*\mathsf D_N)(m)\subseteq ((\phi^*\mathsf D_N)
(m))^\perp$ is easy.
For the converse inclusion, choose $(u_m,\gamma_m)\in ((\phi^*\mathsf D_N)(m))^\perp\subseteq 
T_mM\times T_m^*M$. Since $\ker(T_m\phi:T_mM\to T_{\phi(m)}N)\times \{0_m\}\subseteq (\phi^*\mathsf D_N)(m)$
by definition, we have 
$\gamma_m(v_m)=0$ for all $v_m\in \ker T_m\phi$. Thus, 
$\gamma_m\in(\ker T_m\phi)^\circ=\operatorname{range}(T_m\phi)^*$ and 
there exists $\delta_{\phi(m)}\in T_{\phi(m)}^*N$
such that $\gamma_m=(T_m\phi)^*\delta_{\phi(m)}$.
Then, the equality 
\begin{align*}
0&=\langle (u_m,\gamma_m), (v_m,(T_m\phi)^*\beta_{\phi(m)})\rangle\\
&=\langle (u_m,(T_m\phi)^*\delta_{\phi(m)}), (v_m,(T_m\phi)^*\beta_{\phi(m)})\rangle\\
&=\langle (T_m\phi (u_m),\delta_{\phi(m)}), (T_m\phi(v_m),\beta_{\phi(m)})\rangle
\end{align*}
holds for all pairs $(v_m,(T_m\phi)^*\beta_{\phi(m)})\in (\phi^*\mathsf D_N)(m)$, 
that is,
for all $(T_m\phi v_m,\beta_{\phi(m)})\in\mathsf D_N(\phi(m))$. This
leads to
\begin{align*}
(T_m\phi(u_m),\delta_{\phi(m)})\in 
&\left(\mathsf D_N(\phi(m))\cap(\operatorname{range}
(T_m\phi)\times T_{\phi(m)}^*N)\right)^\perp\\
&=\mathsf D_N(\phi(m))+\{0_{\phi(m)}\}\times \ker(T_m\phi)^*.
\end{align*}
Thus, there exists $\delta'_{\phi(m)}\in T_{\phi(m)}^*N$ such that
$(T_m\phi)^*\delta'_{\phi(m)}=(T_m\phi)^*\delta_{\phi(m)}=\gamma_m$ and 
$(T_m\phi u_m,\delta'_m)\in\mathsf D_N(\phi(m))$. This shows that 
$(u_m,\gamma_m)\in(\phi^*\mathsf D_N)(m)$.

 Hence, if $\phi^*\mathsf D_N$ is smooth,
it is a Dirac structure on $M$ such that $\phi$ is a backward Dirac map.
The Dirac structure $\phi^*\mathsf D_N$ is then the \emph{backward Dirac 
image} of $\mathsf D_N$ under $\phi$.

\subsubsection{Symmetries of Dirac manifolds}
Let $G$ be a Lie group and
$\Phi: G\times M \rightarrow M$ a smooth left action. Then $G$ is 
called a \emph{symmetry Lie group of} $\mathsf{D}$ if for every 
$g\in G$ the condition $(X,\alpha) \in \Gamma(\mathsf{D})$ implies 
that  $\left(\Phi_g^\ast X, \Phi_g^\ast \alpha \right)\in\Gamma(\mathsf{D})$. 
In other words, $\Phi_g:(M,\mathsf D)\to (M,\mathsf D)$ is 
a forward and backward Dirac map for all $g\in G$.
   We say then that the Lie group $G$ acts
\emph{canonically} or that the action of $G$ on $M$ is \emph{Dirac}. 

Let $\mathfrak{g}$ be a Lie algebra and $\xi \in \mathfrak{g} \mapsto \xi_M \in
\mathfrak{X}(M)$ be a smooth left Lie algebra action, that is, the map $$(x, \xi) \in
M \times \mathfrak{g} \mapsto \xi_M(x) \in TM $$ is smooth and $\xi \in\mathfrak{g}
\mapsto \xi_M  \in \mathfrak{X}(M)$ is a Lie algebra anti-homomorphism.  The Lie
algebra $\mathfrak{g}$ is said to be a 
\emph{symmetry Lie algebra of} $\mathsf{D}$ if for every $\xi \in \mathfrak{g}$ 
the condition
$(X,\alpha) \in \Gamma(\mathsf{D})$ implies that  
$\left({\boldsymbol{\pounds}}_{\xi_M}X,{\boldsymbol{\pounds}}_{\xi_M}\alpha \right) \in
\Gamma(\mathsf{D})$.  Of course, if $\mathfrak{g}$ is the Lie algebra of
$G $ and $\xi\mapsto \xi_M$ the  Lie algebra anti-homomorphism, then if $G $
is a symmetry Lie group  of $\mathsf{D}$ it follows that $\mathfrak{g}$ is a 
symmetry Lie
algebra of $\mathsf{D}$.

\subsubsection{Regular reduction of Dirac manifolds}
Let $(M,\mathsf D)$ be a smooth Dirac manifold
with a proper smooth Dirac action of
a Lie group $G$ on it, such that all isotropy subgroups of the action are 
conjugated. Then the space $\bar M:=M/G$ of orbits of the action is 
a smooth manifold and the orbit map $\pi:M\to\bar M$ is a smooth surjective 
submersion. We have the following theorem (see \cite{JoRaZa11}).

\begin{theorem}\label{conj-orbits-red} 
Let $G$ be a connected Lie group acting in a proper way on the manifold 
$M$, such that all isotropy subgroups are conjugated. Assume that 
$\mathsf{D}\cap\K^\perp$
has constant rank on $M$, where $\K^\perp$ is defined as the direct sum 
$TM\operp\V^\circ$.
Then the  Dirac structure $\mathsf{D}$ on $M$ induces a Dirac structure 
$\bar{\mathsf{D}}$ on the quotient $\bar M=M/G$ given by
\begin{equation}\label{red_dir_conj_iso}
\bar{\mathsf{D}}(\bar m)=\left\{\left(\bar X(\bar m),\bar\alpha(\bar m)\right)\in
  T_{\bar m}\bar M\times T^*_{\bar m}\bar M \,\bigg|
\begin{array}{c}\exists X\in
  \mx(M)\text{ such that } X\sim_{\pi}\bar X\\ 
\text{ and }(X,\pi^*\bar\alpha)\in\Gamma(\mathsf{D})\end{array}\right\}
\end{equation} 
for all $\bar{m} \in \bar{M}$. If $(M,\mathsf{D})$ is integrable, 
then $(\bar M,\bar{\mathsf{D}})$ is also integrable.
\end{theorem}
The Dirac structure $\bar{\mathsf D}$ is then the \emph{forward Dirac
image $\pi(\mathsf D)$ of $\mathsf D$ under $\pi$}.

\section{Proper actions and orbit type manifolds}\label{sec:proper}
\subsection{Orbits of a proper action}\label{subsection_orbits}

In this section we consider a left smooth proper action 
\begin{equation}
\begin{array}{cccl}
\Phi &:G\times M&\rightarrow& M\\
&(g,m)&\mapsto& \Phi (g,m)\equiv \Phi
_{g}(m)\equiv gm\equiv g\cdot m  \label{action}
\end{array}
\end{equation}
of a Lie group $G$ on a manifold $M$. Let $\pi :M\rightarrow \bar{M}$ be
the natural projection on the orbit space.

For each closed Lie subgroup $H$ of $G$ we define the \emph{isotropy type} set
\begin{equation*}
M_{H}=\{m\in M\mid G_{m}=H\} ,
\end{equation*}
where $G_{m}=\{g\in G\mid gm=m\}$ is the isotropy subgroup of 
$m\in M$. Since the action is proper, all isotropy subgroups are compact. 
The sets $M_{H}$, where $
H $ ranges over the set of closed Lie subgroups of $G$ for which 
$M_{H}$ is non-empty, form a partition of $M$ and therefore they are the equivalence 
classes of an equivalence relation in $M$.
Define the normalizer of $H$ in $G $
\[
N(H)=\{g\in G\mid gHg^{-1}=H\} 
\]
which is a closed Lie subgroup of $G$. Since $H$ is a normal subgroup 
of $N(H)$, the quotient $N(H)/H$ is a Lie group. If $m\in M_{H}$, 
we have $G_{m}=H$ and $G_{gm}=gHg^{-1}$, for all $g\in G$. 
Therefore, $gm\in M_{H}$ if and only if $g\in N(H)$. The action of $G$ on $M$ restricts to an
action of $N(H)$ on $M_{H}$ which induces a free and proper action of 
$N(H)/H$ on $M_{H}$.

Define the \emph{orbit type} set
\begin{equation}
M_{(H)}=\{m\in M\mid G_{m}\text{ is conjugate to }H\mathbf{\}.}  \label{P(H)}
\end{equation}
Then
\[
M_{(H)}=\{gm\mid g\in G,m\in M_{H}\}=\pi ^{-1}(\pi (M_{H})). 
\]
Connected components of  $M_{H}$ and $M_{(H)}$ are embedded 
submanifolds of $M$;  therefore $M_H$ is called an \emph{isotropy 
type manifold} and $M_{(H)}$ an \emph{orbit type manifold}.
Moreover, 
\[
\pi \left(M_{(H)} \right)=\{gm\mid m\in M_{H}\}/G=M_{H}/N(H)=M_{H}/(N(H)/H). 
\]
Since the action of $N(H)/H$ on $M_{H}$ is free and proper, it 
follows that  $M_{H}/(N(H)/H)$ is a quotient manifold of $M_{H}$. Hence, 
$\pi (M_{(H)})$ is a manifold contained in the orbit space $\bar{M}=M/G$.

For a connected component $Q$ of $M_H$, we denote by $N_Q$ the 
subgroup of $N(H)$ leaving $Q$ invariant, i.e.,
\begin{equation}\label{N_Q}
N_Q:=\{g\in G\mid g\cdot q\in Q \text{ for all } q\in Q\}.
\end{equation}
 Thus $N_Q$ is a union of 
connected components of $N(H)$ and is equal to $N(H)$ if $N(H)$ 
is connected.

Partitions of the orbit space $\bar{M}=M/G$ by connected 
components of $\pi\left(M_{(H)}\right)$ is a decomposition of the 
differential space $\bar{M}$. The
corresponding stratification of $\bar{M}$ is called the orbit type
stratification of the orbit space (see \cite{DuKo00}, 
\cite{Pflaum01}).
It is a minimal stratification in the partial order discussed above (see 
\cite{Bierstone75}). This implies that the strata 
$\pi\left(M_{(H)}\right)$ 
of the orbit type stratification are accessible sets of the family of all 
vector fields  on $\bar{M}$ (see \cite{LuSn08}).

\medskip

The smooth distribution $\T_G\subseteq TM$ spanned by the
$G$-invariant vector fields, i.e.,
\begin{equation}\label{T_G}
\T_G(m):=\left\{X(m)\mid X\in\mx(M)^G,\,\text{such that}\, 
m\in\dom(X)\right\}\quad \text{ for all }
\quad m\in M,
\end{equation}
is shown in \cite{OrRa04} to be completely integrable in 
the sense of Stefan and Sussmann. Its leaves are the 
connected components of the isotropy types.

The smooth distribution $\T\subseteq TM$ is defined as the span of the
descending vector fields, that is, the vector fields $X\in\mx(M)$ 
satisfying $[X,\Gamma(\V)]\subseteq \Gamma(\V)$. A descending 
vector field $X$ can be written as a sum
$X=X^G+X^\V$, with $X^G\in\mx(M)^G$ and $X^\V\in\Gamma(\V)$. 
Therefore,
\begin{equation}\label{T}
\T:=\T_G+\V,
\end{equation}
 $\T$ is completely integrable in the sense of
Stefan and Sussmann, and its integral leaves are the connected
components of the orbit types. 
For the proofs of these statements see \cite{JoRaSn11} and \cite{OrRa04}. 

A local section $(X,\alpha)$ of $TM\oplus
\V^\circ=\K^\perp$ satisfying $[X,\Gamma(\V)]\subseteq
\Gamma(\V)$ and $\alpha\in\Gamma(\V^\circ)^G$ is called a
\emph{descending section} of $\mathsf{P}_M$.

\subsection{Tube theorem and $G$-invariant average}
\label{tubeth}
If the action of the Lie group $G$ on $M$ is proper, we can find for each
point $m\in M$ a $G$-invariant neighborhood of $m$ such that the action can be
described easily on this neighborhood. The proof of the following theorem can
be found, for example,  in \cite{OrRa04}.
\begin{theorem}[Tube Theorem]
Let $M$ be a manifold and $G$ a Lie group
acting properly on $M$. For a given point $m\in M$ denote $H := G_m$. 
Then there exists a 
$G $-invariant open neighborhood $U$ of the orbit $G\cdot m$, called 
\emph{tube at $m$}, and a $G $-equivariant diffeomorphism
$G \times_H B \stackrel{\sim}\longrightarrow  U$. The set $B$ is an open 
$H$-invariant neighborhood of $0$
in an $H $-representation space $H$-equivariantly isomorphic to $T_mM/T_m(G\cdot m)$. 
The $H $-representation on $T_mM/T_m(G\cdot m)$ is given by
$h\cdot (v + T_m(G \cdot m)) := T_m\Phi_h(v) + T_m(G
\cdot m)$, $h \in H $, $v \in T_mM $. The smooth manifold  
$ G \times_H B$ is the quotient of the smooth 
free and proper (twisted)
action $\Psi$ of $H$ on $G\times B$ given by $\Psi(h,(g,b)):=(g
h^{-1},h\cdot b)$, $g \in G $, $h \in H $, $b \in B$. 
The $G $-action on $G \times _H B $ is given by $k\cdot [g, b]: = [kg, b]_H $, 
where $k, g \in G $, $b \in B $, and $[g, b]_H \in G \times _H B $ is the
equivalence class 
{\rm (}i.e., $H $-orbit{\rm )} of $(g,b)$.
\end{theorem}

\subsubsection{$G$-invariant average}
Let $m\in M$ and $H:=G_m$. If the action of $G$ on $M$ is proper, 
the isotropy subgroup $H$ of $m$ is a compact Lie subgroup of $G$. 
Hence, there exists a  Haar measure $\lambda^H $ on $H$, that is, 
a $G$-invariant measure on $H$ satisfying
$\int_Hd\lambda^H=1$ (see, for example, \cite{DuKo00}). Left $G$-invariance of
$\lambda^H$ is equivalent to right $G$-invariance of $\lambda^H$ and 
$R_{h}^*d\lambda^H=d\lambda^H=L_{h}^*d\lambda^H$ for all $h\in H$, where $L_h:H\to H$ 
(respectively
$R_h:H\to H$) denotes left (respectively right) translation by $h$ on $H$.  

Let $X\in\mx(M)$ be defined on the tube $U$ at $m\in M$ for the 
proper action of the Lie group $G$ on $M$. As in the Tube Theorem, 
we write  the points of $U$ as equivalence classes $[g,b]_H$ with 
$g\in G$ and $b\in B$. Recall that for all $h\in H$ we have
$[g,b]_H=[gh^{-1},h\cdot b]_H$. Furthermore, the action of 
$G$ on $U$ is given by $\Phi_{g'}([g,b]_H)=[g'g,b]_H$, for $g' \in G $.
 Define $X_G \in \mathfrak{X} \left(G \times_H B \right)$ by
\begin{align*}
X_G([g,b]_H)&:=\left(\Phi_{g^{-1}}^*\left(\int_H\Phi_h^*Xd\lambda^H\right)
\right)([g,b]_H) \\
&=T_{[e,b]_H}\Phi_g
\left(\int_H\left(T_{[h,b]_H}\Phi_{h^{-1}}X([h,b]_H)\right)d\lambda^H\right)
\end{align*}
for each point $m'=[g,b]_H\in U$.
This definition doesn't depend on the choice of the
representative $[g,b]_H$ for the point $m'$. 
The vector field  $X_G$ is an element of $\mx(M)^G$ (see \cite{JoRaSn11})
and is called the \emph{$G$-invariant average\/} of the vector field
$X$. Note that $X_G$ is, in general, not equal to $X$ (at any point); it can
even vanish. 

Similarly, if $\alpha\in\Omega^1(M)$, define the 
$G$-invariant average $\alpha_G\in \Omega^1(M)^G$ of $\alpha$ by
\begin{align}
\alpha_G([g,b]_H)&:=\left(\Phi_{g^{-1}}^*
\left(\int_H\Phi_h^*\alpha d\lambda^H\right)\right)([g,b]_H)\\
&=\left(\int_H\Phi_h^*\alpha d\lambda^H\right)_{[e,b]_H}\circ
T_{[g,b]_H}\Phi_{g^{-1}}\nonumber\\
&=\left(\int_H(\alpha([h,b]_H)\circ T_{[e,b]_H}\Phi_h) d\lambda^H\right)\circ
T_{[g,b]_H}\Phi_{g^{-1}}\label{alpha_G}
\end{align}
for each point $m'=[g,b]_H\in U$.

If $(X,\alpha)$ is a section of a $G$-invariant generalized
distribution $\D$, the section $(X_G,\alpha_G)$ is a $G$-invariant section
of $\D$. 

If $f$ is a smooth function defined in the tube $U$ of the action of $G $
at $m \in M$, define its $G$-invariant average $f_G$ by
\[
f_G([g,b]_H):=\int_{h\in H}f([h,b]_H)d\lambda^H.
\]

Because the action of $G$ on $M$ is proper, there exists a $G$-invariant 
Riemannian metric $\varrho$ on $M$ (see \cite{DuKo00}).
Let $Q$ be a connected component of an isotropy type manifold $M_H$, 
$H$ compact subgroup of $G$. Then $Q$ is an embedded submanifold 
of $M$. We write  $TM\an{Q}=TQ\oplus TQ^\varrho$ with $TQ^\varrho$
the subbundle of $TM\an{Q}$ orthogonal to $TQ$ (seen as a subbundle of
$TM\an{Q}$) relative to $\varrho$.

\begin{lemma}\label{average_is_projection}
The $\varrho$-orthogonal projection of $v_m\in T_mM$ onto $T_mQ$ is equal to its
$G$-invariant average $\int_HT_m\Phi_h(v_m)\,d \lambda^H$ at $m\in Q$. The
composition
of $\alpha_m\in T_m^*Q$ with $\pr_{T_mQ}$ is equal to the average 
$\int_H(T_m\Phi_h)^*(\alpha_m)\,d \lambda^H$ at $m\in Q$.
\end{lemma}
\begin{proof}
Recall that $\T_G$ is the (completely integrable
in the sense of Stefan and Sussmann) tangent distribution on $M$ that is spanned by the $G$-invariant 
vector fields 
on $M$
and that $Q$ is a leaf of $\T_G$; see \eqref{T_G} and the considerations following it.

The vector $v_m$ can be written as an orthogonal sum $v_m=v_m^\top+v_m^\varrho$ with
$v_m^\top\in T_mQ$ and $v_m^\varrho\in T_mQ^\varrho$.
Then we find smooth (without loss of generality global) 
vector fields $X^\top\in\Gamma(\T_G)$ and $X^\varrho\in\mx(M)$ such that
  $X^\varrho\an{Q}\in\Gamma(TQ^\varrho)$ ($X^\top\an{Q}\in\Gamma(TQ)$ by
the properties of $\T_G$),
and with values $X^\top(m)=v^\top_m$ and $X^\varrho(m)=v_m^\varrho$.
Consider the $G$-invariant averages $X_G^\top$ and $X_G^\varrho$
of $X^\top$ and $X^\varrho$ in a tube centered at $m$. Then 
$X_G^\varrho$ is $G$-invariant, $X_G^\varrho\in\Gamma(\T_G)$, 
 and we get $X_G^\varrho\an{Q}\in\Gamma(TQ)$. But since the metric 
 $\varrho$ is $G$-invariant,
the orthogonal space $TQ^\varrho$ is $G$-invariant and the average 
$X_G^\varrho\an{Q}$ remains a section of $TQ^\varrho$. Hence, it 
must be the zero section.
In the same manner, we have $X_G^\top\in\mx(M)^G$ and thus,
$X_G^\top\an{Q}$ remains tangent to $Q$. In particular,
we get at the point $m=[e,0]_H$:
\begin{align*}
X_G^\top(m)=\int_HT_m\Phi_{h\inv}X^\top(m)d\lambda^H
=\int_HT_m\Phi_{h\inv}(v^\top_m)d\lambda^H=\int_Hv^\top_md\lambda^H=v^\top_m.
\end{align*}
The third equality is proved in the following way. Since $v_m^\top$ is tangent to $Q$, there exists a curve
$c:(-\varepsilon,\varepsilon)\to Q\subseteq M_H$
such that $c(0)=m$ and $\dot c(0)=v_m^\top$.
We have then $T_m\Phi_{h\inv}\left(v_m^\top\right)
=\left.\frac{d}{dt}\right\an{t=0}\left(\Phi_{h\inv}\circ c\right)(0)
=\dot{c}(0)=v_m^\top$ for all $h\in H$.
This leads to $\int_HT_m\Phi_h(v_m)\,d \lambda^H
=X_G^\top(m)+X_G^\varrho(m)=v_m^\top+0=\pr_{T_mQ}(v_m)$.

Choose  now $\alpha_m\in T_m^*M$. Then the $G$-invariant average 
of $\alpha_m$ at $m$ is equal
to
\begin{align*}
\alpha_m=\int_H\alpha_m\circ T_m\Phi_hd\lambda^H
=\alpha_m\circ \int_HT_m\Phi_hd\lambda^H=\alpha_m\circ\pr_{T_mQ},
\end{align*}
as follows from the first statement.
\qed\end{proof}

\section{Induced Dirac structures on the isotropy type manifolds}
Let $(M,\mathsf D)$ be a smooth Dirac manifold
with a smooth proper Dirac action of a Lie group $G$
on it. Let $Q$ be a connected component of an isotropy type 
$M_H$, for a compact subgroup $H\subseteq G$. Then $Q$
is an embedded submanifold of $M$.
We denote by $\iota_Q:Q\hookrightarrow M$ the inclusion.

\subsection{Dirac structures on connected components of isotropy type 
submanifolds}

We show in this subsection that $\mathsf D$ induces a 
Dirac structure $\mathsf D_Q$ on $Q$. Then, we will study 
in the next subsection the induced action
of $N_Q$ (defined by \eqref{N_Q}) on $Q$.

\begin{theorem}\label{thm1}
Define
$\mathsf D_{Q}\subseteq \mathsf P_{Q}$ by
\begin{equation}
\label{dq}
\mathsf D_{Q}(q)
=\left\{
(\tilde v_q, \tilde\alpha_q)\in\mathsf P_Q(q)\left|
\begin{array}{c}
\exists\, (v_q,\alpha_q)\in\mathsf D(q)\,\text{ such that }\\
T_q\iota_Q\tilde v_q=v_q \text{ and }
(T_q\iota_Q)^*\alpha_q=\tilde\alpha_q
\end{array}\right.\right\}
\end{equation}
for all $q\in Q$, i.e., $\mathsf D_Q$ is the \emph{backward Dirac image
of $\mathsf D$ under $\iota_Q$}.
Then $\mathsf D_Q$ is a Dirac structure on $Q$.
If $(M,\mathsf D)$ is integrable, then $(Q,\mathsf D_Q)$ is integrable.
\end{theorem}

\begin{proof}
Since $\mathsf D_Q$ is, by definition,  equal
to $\iota_Q^*\mathsf D$ (see \eqref{pullback} and 
the considerations following this equation), 
we have only to check that  
 \eqref{dq} defines a smooth generalized distribution in $\mathsf P_Q$.
Choose $q\in Q$ and $(\tilde v_q, \tilde\alpha_q)\in\mathsf D_Q(q)$.
For simplicity, we also write $q$ for $\iota_Q(q)\in M$.
 Then we find 
$(v_q,\alpha_q)\in\mathsf D_Q(q)$ such that 
$T_q\iota_Q\tilde v_q=v_q\in T_qQ\subseteq T_qM$ and 
$(T_q\iota_Q)^*\alpha_q=\tilde\alpha_q$. Since $\mathsf D$ is a smooth vector bundle 
on $M$, there exists $(X,\alpha)\in\Gamma(\mathsf D)$ with $q\in\dom(X,\alpha)$ such that
$(X,\alpha)(q)=(v_q,\alpha_q)$.
Consider the $G$-invariant average
$(X_G,\alpha_G)$ of the
pair $(X,\alpha)$ in a tube in $M$ centered at $q$. Since $\mathsf D$ is
invariant under the action of $G$ on $M$, we have
$(X_G,\alpha_G)\in\Gamma(\mathsf D)$ and since it is $G$-invariant,
$X_G\an{Q}$ is tangent to 
the connected component $Q$ of the isotropy type manifold $M_H$. 
This shows that $X_G\in\mx(M)$ is such that there exists $\tilde X\in\mx(Q)$ with
$\tilde X\sim_{\iota_Q} X_G$. By definition of $\mathsf D_Q$, the pair
$(\tilde X, \iota_Q^*\alpha_G)$ is a section of $\mathsf D_Q$.
Furthermore, we have 
\begin{align*}
X_G(q)&=\int_H(\Phi_h^*X)(q)d\lambda^H=\int_HT_q\Phi_{h^{-1}}X(q)d\lambda^H\\
&=\int_HT_q\Phi_{h^{-1}}(v_q)d\lambda^H=\pr_{T_qQ}(v_q)=v_q
\end{align*}
by Lemma \ref{average_is_projection} since $v _q\in T_qQ$. 
This leads to $\tilde X(q)=\tilde v_q$.
In the same manner,
$\alpha_G(q)=\alpha_q \circ \pr_{T_qQ} = \alpha_q$ since $\alpha_q 
\in \left((T_qQ)^\varrho \right) ^ \circ $
and thus $(\iota_Q^*\alpha_G)(q)=\alpha_G(q)\circ T_q\iota_Q=\alpha_q\circ T_q\iota_Q=\tilde\alpha_q$.
 Hence, we have found a smooth section
$(\tilde X,\iota_Q^*\alpha_G)$ of 
$\mathsf D_Q$ whose value at $q $ is  
$(\tilde v_q,\tilde \alpha_q)$. 

\medskip

Since $\iota_Q:(Q,\mathsf D_Q)\to(M,\mathsf D)$ is a backward Dirac map, we know using 
for instance Lemma 2.2 in \cite{StXu08} that 
$\mathsf D_Q$ is integrable if $\mathsf D$ is integrable. 
\qed\end{proof}

\begin{remark}
In the situation of the previous theorem,
one can show with the same methods as in the proof of the smoothness of $\mathsf D_Q$ that the intersection 
\begin{equation}\label{intersection}
\mathsf D\cap(TQ\operp(TQ^\varrho)^\circ)
\end{equation} is smooth. It has hence constant 
rank on $Q$ by Proposition \ref{prop_intersection}. This intersection
is then a Dirac structure in $TQ\operp(TQ^\varrho)^\circ$. 
The Dirac structure $\mathsf D_Q$ can be seen as the pullback of this intersection 
via the identification of $\mathsf P_Q$ with $TQ\operp(TQ^\varrho)^\circ$.

Assume that $N$ is an embedded submanifold of a manifold $M$ endowed with a Dirac structure
$\mathsf D_M$. It was already shown in \cite{Courant90a} that 
if $\mathsf D_M\cap(TN\operp T^*M\an{N})$ has constant rank, then 
there is an induced
Dirac structure on $N$ defined as in \eqref{dq} and
such that the inclusion map $N\hookrightarrow M$ is a backward Dirac map.
The hypothesis in \cite{Courant90a}
ensures that the bundle 
defined by \eqref{dq} is smooth. In the present situation
we cannot use this known
result of \cite{Courant90a} for the proof of 
Theorem \ref{thm1}. On the other hand, averaging techniques  
prove that the bundle defined by \eqref{dq} is smooth. We could also have shown first that \eqref{intersection}
is smooth (and has hence constant rank) and then  that its pullback to $Q$ is a Dirac structure on $Q$.
\end{remark}

\subsection{Induced Dirac structures on the quotients}
Let $G $ act on the Dirac manifold $(M,\mathsf D)$ properly and 
canonically. Let $Q$ be a connected component of the orbit type
manifold $M_H$, for $H$ a compact subgroup of $G$.  
In Theorem \ref{thm1}, we have shown that there is an induced Dirac structure 
defined by \eqref{dq} on $Q$ such that $\iota_Q:(Q,\mathsf D_Q)\to (M,\mathsf D)$ is a backward Dirac map.
We will show here that the action of $N_Q$ on $(Q,\mathsf D_Q)$ is a  proper Dirac action 
that satisfies the conditions for regular reduction.

From the proof of Theorem \ref{thm1}, 
we can see that each pair $(\tilde v_q,\tilde\alpha_q)\in\mathsf D_Q(q)$
corresponds to a unique pair $(v_q,\alpha_q)\in\mathsf D(q)\cap(T_qQ\cap(T_qQ^\varrho)^\circ)$ such that
$T_q\iota_Q\tilde v_q=v_q$ and $\tilde \alpha_q=(T_q\iota_Q)^*\alpha_q$.
The converse is also true by definition of $\mathsf D_Q(q)$.
Thus, the map
$$ 
\begin{array}{lcclc}
I_q:&\mathsf D_Q(q)
&\longrightarrow& \mathsf D(q)\cap\left(T_qQ\times\left(T_qQ^\varrho\right)^\circ\right)
&(\subseteq T_qM\times T_q^*M)\end{array}$$
sending each $(\tilde v_q,\tilde\alpha_q)$ to the corresponding $(v_q,\alpha_q)$
is an isomorphism of vector spaces 
for all $q\in Q$.

We will use the maps $I_q$, $q\in Q$, as a technical tool in the proof of the following lemma. 
Recall that $\T$ is the smooth tangent distribution
defined as the sum of $\T_G$ and $\V$ (see \eqref{T_G} and \eqref{T}).

\begin{lemma}\label{intersection_has_constant_rank_cor}
Let $N_Q$
be the subgroup of $N(H)$ leaving $Q$ invariant 
(see \eqref{N_Q}), 
denote by $\V_Q$ the vertical distribution  of the induced proper action of $N_Q$ on $Q$, and
set $\K_Q:=\V_Q\operp\{0\}\subseteq \mathsf P_Q$.
If $\mathsf D\cap(\T\operp\V_G^\circ)$ is smooth, then 
the intersection $\mathsf D_Q\cap\K_Q^\perp$ has constant rank on $Q$.
\end{lemma}

\begin{proof}
Consider the restriction of the map $I_q$ defined above
to the set $\mathsf D_Q(q)\cap\K_Q^\perp(q)$.
We show that if $\tilde
\alpha_q\in(\V_Q)^\circ(q)$,
then $\alpha_q\in\V^\circ_G(q)$. 
Since the action of $N_Q$ on $Q$ is  proper with fixed isotropy $H $, 
the vertical space $\V_Q$ is a smooth vector bundle on $Q$,
and we have $(\V_Q)^\circ(q)=(\V_Q(q))^\circ=(T_q(N_Q\cdot q))^\circ$.
But since the action of $H$ on $Q$ is trivial, we have
also  $\left((T_q(N_Q\cdot q))^\circ\right)^H=(T_q(N_Q\cdot q))^\circ$.
Using Theorem 2.5.10 in \cite{OrRa04}
we find then 
\[(\V_Q)^\circ(q)=\left((T_q(N_Q\cdot q))^\circ\right)^H
=\left\{\dr f(q)\mid f\in C^\infty(Q)^{N_Q} \right\}.
\]
If $\tilde\alpha_q\in(\V_Q)^\circ(q)$, we find hence 
a smooth $N_Q$-invariant function $\tilde f\in C^\infty(Q)^{N_Q}$ such that
$\dr \tilde f_q=\tilde \alpha_q$. 
Since $Q$ is an embedded submanifold of $M$, there exists 
then a smooth function $f\in C^\infty(M)$ such that
$\iota_Q^*f=\tilde f$. Consider the $G$-invariant average 
$F$ of $f$ at $q$. Since $\tilde f$ was $N_Q$-invariant,
we still have $\iota_Q^*F=\tilde f$. Therefore, 
$\dr F_q\circ T_q\iota_Q=(\iota_Q^*\dr F)(q)=\dr \tilde f_q $
and 
$\dr F\in\Gamma(\V^\circ)^G$, that is, $\dr F_q\in\V_G^\circ(q)$.
Also, since $\dr F$ is $G$-invariant, we have 
$\dr F_q=\dr F_q\circ \pr_{TqQ}$ by Lemma \ref{average_is_projection} and thus
$\dr F_q\in 
\left(T_qQ^\varrho\right)^\circ$. The covector 
$\dr F_q$ is hence the unique element of $\left(T_qQ^\varrho\right)^\circ$
satisfying $\tilde\alpha_q=\dr F_q\circ T_q\iota_Q$. Thus $\dr F_q$ is equal
to the covector $\alpha_q$ and we have shown that $\alpha_q\in\V_G^\circ(q)$. Hence we get 
$$
I_q\an{\mathsf D(q)\cap\K_Q^\perp(q)}:
\mathsf D_Q(q)\cap\K_Q^\perp(q)\to \mathsf D(q)\cap(T_qQ\times \V_G^\circ(q)).
$$
This map is obviously surjective since  $ \V_G^\circ\an{Q}\subseteq (TQ^\varrho)^\circ$
because $\V_G^\circ$ is spanned by $G$-invariant sections.

\medskip

Since $\mathsf D\cap (\T\operp\V_G^\circ)$ is smooth by hypothesis, it follows
by $G$-invariant averaging that
$\mathsf D\cap(\T_G\operp\V_G^\circ)$ is also smooth (see  \cite{JoRa10b}).

We use this to show  that $\mathsf D_Q\cap\K_Q^\perp$ is smooth. Choose
$q\in Q$ and $(v_q,\alpha_q)\in\mathsf  D_Q(q)\cap\K_Q^\perp(q)$. 
Then $I_q(v_q,\alpha_q)\in 
\mathsf D(q)\cap\left(T_qQ\operp\V_G^\circ(q)\right)$. Since
$\mathsf D\cap\left(TQ\operp\V_G^\circ\an{Q}\right)=
\mathsf D\cap(\T_G\operp\V_G^\circ)\an{Q}$, we find
a smooth section $(X,\alpha)\in\Gamma(\mathsf D\cap(\T_G\operp\V_G^\circ))$
defined on a neighborhood $U$ of $q$ in  $M$ 
such that $(X,\alpha)(q)=I_q(v_q,\alpha_q)$.
Since $(X,\alpha)$ is a section of 
$\mathsf D$ restricting to a section 
of $\mathsf D\cap\left(TQ\operp\V_G^\circ\an{Q}\right)
\subseteq \mathsf D\cap\left(TQ\operp\left(TQ^\varrho\right)^\circ\right)$
on $Q$, we get 
by the proof of Theorem \ref{thm1}
the existence
of a smooth section $(\tilde X,\tilde \alpha)$ of 
$\mathsf D_Q$ such  that
$\tilde X\sim_{\iota_Q}X$ and $\tilde \alpha=\iota_Q^*\alpha$.
But for all $q'\in Q\cap U$, we get
from the considerations above that
$(\tilde X,\tilde \alpha)(q')
=I_{q'}\inv\left((X,\alpha)(q')\right)\in\mathsf D_Q(q')\cap\K_Q^\perp(q')$.
Thus, $(\tilde X,\tilde \alpha)$ is a smooth section of
$\mathsf D_Q\cap\K_Q^\perp$, taking the value
$(\tilde X,\tilde \alpha)(q)=(v_q,\alpha_q)$
at $q$. 

Since $\mathsf D_Q\cap\K_Q^\perp$ is smooth
and $\mathsf D_Q$ and 
$\K_Q^\perp$ are the smooth orthogonal bundles  
of the vector subbundles 
$\mathsf D_Q$ and 
$\K_Q$ of $(\mathsf P_Q,\pairing_{\mathsf P_Q})$, respectively, we can conclude
using  Proposition \ref{prop_intersection}
that $\mathsf D_Q\cap\K_Q^\perp$ has constant rank on $Q$.
\qed\end{proof}
\begin{lemma}\label{thm2}
Let $(M,\mathsf D)$ be a smooth Dirac manifold with a proper smooth Dirac
action of  a Lie group $G$ on it  such that
the intersection $\mathsf D\cap(\T\operp\V_G^\circ)$ is smooth.
The Dirac structure
$\mathsf D_Q$ defined by $\mathsf D$ on $Q$,
as in Theorem \ref{thm1}, is invariant
under the induced action of $N_Q$ (see\eqref{N_Q}) on $Q$
and has the property that
$\mathsf D_{Q}\cap\K_Q^\perp$ 
has constant rank on $Q$.
\end{lemma}

\begin{proof}
Recall the construction of $\mathsf D_{Q}$. Since $\mathsf D$ is 
$G$-invariant,
it is $N_Q$-invariant. The connected component $Q$
 of the isotropy type manifold $M_H$ is an 
accessible set of the family of the $G$-invariant vector fields on $M$. 
Since $\varrho$ is also $G$-invariant,
the spaces $TQ$ and $\left(TQ^\varrho\right)^\circ$ are $N_Q$-invariant 
and so the induced action of $N_Q$ on $(Q,\mathsf D_Q)$ is canonical.

The second claim has been shown in Lemma \ref{intersection_has_constant_rank_cor}.
\qed\end{proof}

\begin{theorem}\label{Reduction_theorem}
Let $(M,\mathsf D)$ be a smooth Dirac manifold with a proper smooth Dirac
action of  a Lie group $G$ on it  such that
the intersection $\mathsf D\cap(\T\operp\V_G^\circ)$ is smooth.
Let $\mathsf D_Q$ be the induced Dirac structure 
on the connected component $Q$ of the isotropy type manifold $M_H$ and 
$q_Q:Q\to Q/N_Q$ the projection (a smooth surjective submersion).
The forward Dirac image $q_Q(\mathsf D_Q)$ {\rm (}as in Theorem 
\ref{conj-orbits-red}{\rm )}  is a Dirac structure
on $Q/N_Q$. If $\mathsf D$ is integrable, then $q_Q(\mathsf D_Q)$ is 
integrable.
\end{theorem}
\begin{proof}
By Lemma \ref{thm2}, all the hypotheses for regular reduction are satisfied
for the smooth proper Dirac action of $N_Q$ on 
$(Q,\mathsf D_{Q})$ with fixed isotropies. Thus the quotient space
$Q/N_Q$ inherits a smooth Dirac structure defined by the
forward Dirac image $q_Q(\mathsf D_{Q}) $ of
$\mathsf D_Q$ on $Q/N_Q$ (see the paragraph about regular Dirac reduction at
the end of Subsection \ref{sec:Dirac_structures}). 
\qed\end{proof}

\subsection{Comparison with the Dirac strata of the reduced space 
$(\bar M,\bar\D)$}
We want to compare the Dirac manifolds $(Q/N_Q,q_Q(\mathsf D_Q))$
obtained above
 with the Dirac manifolds induced 
by the singular Dirac reduction method in \cite{JoRaSn11}.

We present a short review of the Dirac reduction methods in 
\cite{JoRaSn11}.
Let $(M,\mathsf{D})$ be a smooth Dirac manifold acted upon in a smooth proper
and Dirac
manner by a Lie group $G$ such that the intersection $\mathsf
D\cap(\T\operp\V_G^\circ)$ is spanned by its descending sections
(recall that $\T$  is defined by  \eqref{T}).
Let $\pi:M\to M/G$ be the projection.

Consider the subset $\D^G$ of $\Gamma(\mathsf{D})$ defined by
\[ \D^G:=\{(X,\alpha)\in\Gamma(\mathsf{D})\mid \alpha\in\Gamma(\V^\circ)^G\text{ and }
[X,\Gamma(\V)]\subseteq\Gamma(\V)\},\]
that is, the set of the descending sections of $\mathsf{D}$.

Each vector field $X$ satisfying
$[X,\Gamma(\V)]\subseteq\Gamma(\V)$ pushes-forward to a ``vector field'' $\bar X$
on $\bar M$. (Since we will not need these objects in the rest of the paper,
we  will not give more details about 
what we call the ``vector fields'' and ``one-forms''
on the stratified space $\bar M=M/G$ and refer to \cite{JoRaSn11} for more information.)
For each stratum ${\bar
  P}$ of $\bar M$, 
the restriction of $\bar X$ to points
of ${\bar P}$ is a vector field $X_{\bar P}$ on ${\bar P}$. 
On the other hand, if
$(X,\alpha)\in\D^G$, then we have $\alpha\in\Gamma(\V^\circ)^G$
 which pushes-forward to the one-form $\bar\alpha:=\pi_*\alpha$ such that,
for every $\bar Y\in \mx(\bar M)$ and every vector field  $Y\in\mx(M)$ satisfying
$Y\sim_\pi \bar Y$, we have 
\[\pi^*(\bar\alpha(\bar Y))=\alpha(Y).\]
Moreover, for each stratum ${\bar P}$ of $\bar M$, the restriction of $\bar\alpha$ to
points of ${\bar P}$ defines a one-form $\alpha_{\bar P}$ on ${\bar P}$. Let 
\[\bar\D=\{(\bar X,\bar\alpha)\mid (X,\alpha)\in\D^G\}\]
and for each stratum ${\bar P}$ of $\bar M$, set 
\[\D_{\bar P}=\{(X_{\bar P},\alpha_{\bar P})\mid (\bar X,\bar\alpha)\in\bar\D\}.\]
Define the smooth distribution $\mathsf{D}_{\bar P}$ on ${\bar P}$ by
\begin{equation}\label{def_D_barP}
\mathsf{D}_{\bar P}(\bar p)=\{(X_{\bar P}(\bar p),\alpha_{\bar P}(\bar p))\in T_{\bar p}{\bar P}\times
T^*_{\bar p}{\bar P}\mid (X_{\bar P},
\alpha_{\bar P})\in\D_{\bar P}\}
\end{equation}
for all $\bar p\in\bar P$.
Note that $\Gamma(\mathsf{D}_{\bar{P}}) = \mathcal{D}_{\bar{P}}$. We have the 
following theorem (see \cite{JoRaZa11} for the regular case).

\begin{theorem}\label{singred}
Let $(M,\mathsf{D})$ be a Dirac manifold with a proper Dirac action of a connected 
Lie group 
$G$ on it.
Let $\bar P$ be a stratum of the quotient space $\bar M$.
If $\mathsf{D}\cap(\T\operp\V_G^\circ)$ is spanned by its descending sections, then
$\mathsf{D}_{\bar P}$ defined in \eqref{def_D_barP} is a Dirac structure on
${\bar P}$.
If $(M,\mathsf D)$ is integrable, then $(\bar P,\mathsf D_{\bar P})$ is integrable.
\end{theorem}

Note that if $\mathsf D\cap(\T\operp\V_G^\circ)$ is spanned by its descending sections, it is smooth
and all the hypotheses of Theorem \ref{Reduction_theorem} are satisfied.
We can thus compare the reduction methods in 
Theorems \ref{singred} and  \ref{Reduction_theorem}.

We have the following theorem, comparing the
Dirac manifolds $(Q/N_Q,q_Q(\mathsf D_Q))$ and $(\bar P,\mathsf D_{\bar
  P})$ if $Q$ is a connected component of $M_H$, $H$ a compact subgroup of
$G$, and $P$ is the connected component of $M_{(H)}$ containing $Q$. 
\begin{theorem}
Let $(M,\mathsf D)$ be a smooth Dirac manifold
with a proper smooth Dirac action of a Lie group $G$ on it,
such that the intersection $\mathsf D\cap(\T\operp\V_G^\circ)$ is spanned by 
its descending sections.

Let $Q$ be  a connected component of $M_H$, $H$ a compact subgroup of
$G$, and $P$ the connected component of $M_{(H)}$ containing $Q$.
Consider the Dirac manifolds 
 $(Q/N_Q,q_Q(\mathsf D_Q))$ as in Theorem \ref{Reduction_theorem} and 
$(\bar P,\mathsf D_{\bar
  P})$ as in Theorem \ref{singred}, and define the map
$\Phi:Q/N_Q\to \bar P$ by $(\Phi\circ q_Q)(q)=\pi(q)$ for all $q\in
Q\subseteq P$.
Then the map $\Phi$ is a diffeomorphism preserving the Dirac structure,
i.e, it is a forward and, equivalently, a backward Dirac map. 
\end{theorem}

\begin{proof}
We have the following commutative diagram:
\begin{displaymath}
\begin{xy}
\xymatrix{
Q\ar[d]_{q_Q}\ar[r]^{\iota_{Q,P}}\ar@/^0.6cm/[rr]^{\iota_Q}
&P\ar[d]_{\pi\an{P}}\ar[r]^{\iota_P}&M\ar[d]_{\pi}\\
Q/N_Q\ar[r]_\Phi&\bar P\ar[r]_{\iota_{\bp}}&\bar M
}
\end{xy}
\end{displaymath}
where $\iota_{Q, P}, \iota_P, \iota_{\bar{P}}$ are inclusions.
We show that $\Phi$ is bijective. Let $q_Q(m_1)$, $q_Q(m_2)\in Q/N_Q$ 
be such that $\pi(m_1)=\Phi(q_Q(m_1))=\Phi(q_Q(m_2))=\pi(m_2)$. Then there exists
$g\in G$ such that $gm_1=m_2$. Since $m_1,m_2\in Q$, we have 
$gHg\inv=H$ and thus $g\in N(H)$. However, $g $ maps $Q $ 
onto a connected component of $M_H$ containing $m_2 \in Q$, so it follows that $g \in N_Q $. We have then 
$q_Q(m_1)=q_Q(m_2)$ thereby showing that $\Phi$ is injective. Choose
 now $\pi(m) \in \bar P$ with $m\in P$. Since $P\subseteq G\cdot  Q$, there exists
$g\in G$ and $m'\in Q\subseteq P$ such that
$gm'=m$. We get
$(\Phi\circ q_Q)(m')=\pi(m')=\pi(m)$ and hence $\Phi$ is surjective.

\medskip

To show that $\Phi $ is a diffeomorphism, choose a smooth
function $\bar f\in C^\infty(\bar P)$. Then the $G_P$-invariant function
$(\pi\an P)^*\bar f=f$ satisfies
 $f\in C^\infty(P)$, where $G_P$ is the Lie subgroup
of $G$ that leaves $P$ invariant. Since $Q$ is a smooth submanifold 
of $P$, the function $\iota_{Q,P}^*f$ is an element 
of $C^\infty(Q)$. But since $f\in C^\infty(P)^{G_P}$, we have
$\iota_{Q,P}^*f\in C^\infty(Q)^{N_Q}$ and there
exists $\tilde f\in C^\infty(Q/N_Q)$ such that
$q_Q^*\tilde f=\iota_{Q,P}^*f$. We have then 
$q_Q^*(\Phi^*\bar f)=
(\Phi\circ q_Q)^*\bar f=(\pi\an{P}\circ\iota_{Q,P})^*\bar f
=\iota_{Q,P}^*f=q_Q^*\tilde f$ 
and thus
$\Phi^*\bar f=\tilde f\in C^\infty(Q/N_Q)$ which shows that $\Phi$ is smooth.

Choose now a smooth function
$\tilde f\in C^\infty(Q/N_Q)$. Then the pull back $f_Q=q_Q^*\tilde f$
is a $N_Q$-invariant element of $C^\infty(Q)$ and we find, 
by the same method as 
in the proof of Lemma \ref{intersection_has_constant_rank_cor},
a smooth $G$-invariant function $f\in C^\infty(M)^G$ such that
$\iota_Q^*f=f_Q$. We have then 
$f_P:=\iota_P^*f\in C^\infty(P)^{G_P}$ and $\iota_{Q,P}^*f_P=\iota_{Q,P}^*(\iota_P^*f)
=\iota_Q^*f=f_Q$. There exists
$\bar f\in C^\infty(\bar P)$ such that
$(\pi\an{P})^*\bar f=f_P$ and we get
$q_Q^*(\Phi^*\bar f)=  
(\Phi\circ q_Q)^*\bar f=(\pi\an{P}\circ\iota_{Q,P})^*\bar f
=\iota_{Q,P}^*f_P=f_Q=q_Q^*\tilde f$, that
is, $\Phi^*\bar f=\tilde f$ and thus
$(\Phi\inv)^*\tilde f=\bar f\in C^\infty(\bar P)$. Therefore $\Phi^{-1}$ is smooth.

\medskip

We show now that $\mathsf D_{\bar P}$ is the $\Phi$-forward Dirac image of $q_Q(\mathsf D_Q)$.
The fact that $q_Q(\mathsf D_Q)$ is the $\Phi$-backward 
Dirac image of $\mathsf D_{\bar P}$ follows because $\Phi$ 
is a diffeomorphism.

Consider the forward Dirac image  $\Phi(q_Q(\mathsf D_Q))$
of $q_Q(\mathsf D_Q)$ under $\Phi$, defined on $\bar P$ by
$$\Phi(q_Q(\mathsf D_Q))(\bar p)=\left\{ 
(v_{\bar p}, \alpha_{\bar p})\in\mathsf P_{\bar P}(\bar p)\left|\begin{array}{c}
\exists (v_{\bar q}, \alpha_{\bar q})\in q(\mathsf D_Q)(\bar q)
\text{ such that } \\
T_{\bar q}\Phi v_{\bar q}=v_{\bar p} \text{ and } \alpha_{\bar q}=(T_{\bar q}\Phi)^*\alpha_{\bar p}
\end{array}
\right.\right\}
$$
for all $\bar p\in \bar P$ and $\bar q=\Phi\inv(\bar p)$. Since 
$\Phi$ is a diffeomorphism and $q_Q(\mathsf D_Q)$ is a Dirac structure on $Q/N_Q$, 
the forward Dirac image $\Phi(q_Q(\mathsf D_Q))$ is a Dirac structure on $\bar P$.
 Hence, it is sufficient to show 
the inclusion $\mathsf D_{\bar P}\subseteq \Phi(q_Q(\mathsf D_Q))$.

Choose
$(\bar  X,\bar \alpha)\in\Gamma(\mathsf D_{\bar P})$. Then there exists
a pair $(X_{\bar M},\alpha_{\bar M})\in\bar\D$
such that $X_{\bar M}\an{\bar P}=\bar X$ and 
$\alpha_{\bar M}\an{\bar P}=\bar{\alpha}$.
Thus we find $(X,\alpha)\in\D^G$ satisfying $X\sim_\pi X_{\bar M}$ and
$\pi^*\alpha_{\bar M}=\alpha$. The one-form $\alpha$ is a $G$-invariant
section of $\V^\circ$ and since $[X, \Gamma(\mathcal{V})] \subseteq 
\Gamma( \mathcal{V}) $, the vector field $X$ can be written as a sum 
$X=X^G+V$ with $X^G\in\mx(M)^G$ and $V\in\Gamma(\V)$ (see 
\cite{JoRaSn11}). 
Let $(X_G,\alpha_G)$ be the $G$-invariant average of $(X,\alpha)$ in a 
tube $U$ centered at a point $m\in P$. We have then $\alpha_G=\alpha$ 
and $X_G=X^G+V_G$, that is, we still have $X_G\sim_\pi X_{\bar M}$. 
Since $(X_G,\alpha)$ is a $G$-invariant
section of $\mathsf D\cap(\T_G\operp\V_G^\circ)$, we have
$(X_G,\alpha)\an{Q}\in\Gamma\left(\mathsf
D\cap\left(TQ\operp\left(TQ^\varrho\right)^\circ\right)\right)$
and, by definition of $\mathsf D_Q$, we find $(X_Q,\alpha_Q)\in
\Gamma(\mathsf D_Q)$ such that
$X_Q\sim_{\iota_Q}X_G$ and $\alpha_Q=\iota_Q^*\alpha$. The pair 
$(X_Q,\alpha_Q)$ is then automatically a $N_Q$-invariant section 
of $\mathsf{D}_Q\cap\K_Q^\perp$ (by the proof of Lemma 
\ref{intersection_has_constant_rank_cor})
and descends thus to a section $(\widetilde X,
\widetilde \alpha)$ of $ q_Q(\mathsf D_{Q})$, that is,
we have  $X_Q\sim_{q_Q}\widetilde X$ and 
$q_Q^*\widetilde\alpha=\alpha_Q$.
Therefore, for all $m'\in Q$ we have
\begin{align*}
T_{q_Q(m')}\Phi\left(\widetilde X(q_Q(m'))\right)
&=T_{q_Q(m')}\Phi \left(T_{m'}q_Q\left(X_Q(m')\right) \right)\\
&=T_{m'}(\pi\an{P}\circ\iota_{Q,P})\left(X_Q(m')\right)\\
&=T_{m'}(\pi\circ\iota_{Q})\left(X_Q(m')\right)
=T_{m'}\pi \left(X_G(m')\right)\\
&=X_{\bar M}(\pi(m'))=\bar X(\pi(m')),
\end{align*}
that is, $\widetilde X\sim_{\Phi}\bar X$. In the same manner, we show the equality
$q_Q^*\widetilde \alpha=q_Q^*(\Phi^*\bar\alpha)$
which implies that $\widetilde \alpha 
= \Phi^\ast\bar{\alpha}$ since $q_Q:Q \rightarrow Q/N_Q$ is a surjective submersion. Thus, 
$(\bar  X,\bar \alpha)$ is a section of 
$\Phi(q_Q(\mathsf D_Q))$.
\qed\end{proof}

\def\cprime{$'$} \def\polhk#1{\setbox0=\hbox{#1}{\ooalign{\hidewidth
  \lower1.5ex\hbox{`}\hidewidth\crcr\unhbox0}}}
\providecommand{\bysame}{\leavevmode\hbox to3em{\hrulefill}\thinspace}
\providecommand{\MR}{\relax\ifhmode\unskip\space\fi MR }
\providecommand{\MRhref}[2]{%
  \href{http://www.ams.org/mathscinet-getitem?mr=#1}{#2}
}
\providecommand{\href}[2]{#2}

\end{document}